\documentclass[reqno,12pt]{amsart}

\usepackage[margin=1in]{geometry}
\usepackage{amsmath,amssymb,amsthm,mathtools}
\usepackage{bm}
\usepackage{enumitem}
\usepackage[hidelinks]{hyperref}
\usepackage[nameinlink,capitalize]{cleveref}
\usepackage{microtype}

\hypersetup{
  pdftitle={Lewy-Type Nondegeneracy for Gradient-Dependent Elliptic Equations and Global Gradient Diffeomorphisms on Convex Rings},
  pdfauthor={Da Liu, Shujun Shi},
  pdfsubject={Lewy-type nondegeneracy and global gradient diffeomorphisms for gradient-dependent elliptic equations},
  pdfkeywords={Lewy theorem, Hessian nondegeneracy, gradient homeomorphism, gradient-dependent elliptic equation, convex ring, level-set convexity}
}

\allowdisplaybreaks
\newtheorem{theorem}{Theorem}[section]
\newtheorem{proposition}[theorem]{Proposition}
\newtheorem{lemma}[theorem]{Lemma}
\newtheorem{corollary}[theorem]{Corollary}
\newtheorem{remark}[theorem]{Remark}

\newcommand{\R}{\mathbb R}

\newcommand{\Sph}{\mathbb S}
\newcommand{\Sym}{\operatorname{Sym}}
\newcommand{\tr}{\operatorname{tr}}
\newcommand{\rank}{\operatorname{rank}}

\newcommand{\cA}{\mathcal A}
\newcommand{\cQ}{\mathcal Q}
\newcommand{\eps}{\varepsilon}
\newcommand{\dd}{\,\mathrm d}
\numberwithin{equation}{section}
\newcommand{\autotag}{\refstepcounter{equation}\tag{\theequation}}

\title[Lewy-Type Nondegeneracy for Gradient-Dependent Elliptic Equations]{Lewy-Type Nondegeneracy for Gradient-Dependent Elliptic Equations
	and Global Gradient Diffeomorphisms on Convex Rings}
	
\author{Da Liu}
\address{School of Mathematical Sciences, Harbin Normal University, Harbin 150025, Heilongjiang Province, China.}
\email{2024400034@stu.hrbnu.edu.cn}

\author{Shujun Shi}
\address{School of Mathematical Sciences, Harbin Normal University, Harbin 150025, Heilongjiang Province, China.}
\email{shjshi@hrbnu.edu.cn}
\date{}

\begin{document}

\begin{abstract}
We prove two complementary Lewy-type theorems for elliptic equations whose
coefficients depend only on the gradient.  First, let $\Omega\subset\R^3$ and let
$u$ solve
\[
 a^{ij}(Du)u_{ij}=0,
\]
where $a$ is a smooth, symmetric, positive definite matrix field on an open
set containing $Du(\Omega)$.  We show that if the gradient map $Du$ is a local
homeomorphism, then $\det D^2u$ never vanishes; hence $Du$ is a local
$C^\infty$-diffeomorphism.  Second, in every dimension, we consider
capacitary solutions on convex rings.
If the solution has no critical points and its level hypersurfaces are strictly
convex, then ellipticity alone forces the Hessian to have one positive and
$n-1$ negative eigenvalues, that is, inertia $(1,n-1)$.
Consequently, the gradient is a global diffeomorphism onto a radially parametrized
ring in gradient space.  Both results apply to the $p$-Laplace and
minimal surface equations.  For the global minimal-surface result, existence
of a smooth solution is assumed.  Classical convex-ring results supply the noncriticality and strict
level-set convexity needed in the global corollaries.  Since the two coefficient
matrices are real analytic on the relevant gradient ranges, the corresponding
local and global gradient diffeomorphisms are real analytic.
\end{abstract}

\medskip

\maketitle
\section{Introduction}
The relation between topological injectivity and analytic nondegeneracy is one of
the oldest themes in the theory of harmonic mappings.  In 1936, Lewy proved
that a one-to-one harmonic mapping in the plane has nonvanishing Jacobian
\cite{Lewy1936}.  Thus a topological homeomorphism whose coordinate functions are
harmonic is automatically a local real-analytic diffeomorphism.  This theorem is
closely related to, but logically distinct from, the Rad\'o--Kneser--Choquet
theorem: the latter supplies global injectivity from convex boundary data, whereas
Lewy's theorem gives differential nondegeneracy under an injectivity assumption.

For a harmonic function $u$, the natural associated mapping is its gradient
$Du$.  In dimension three, Lewy established the remarkable analogue that a
homeomorphic harmonic gradient has nonvanishing Hessian determinant
\cite{Lewy1968}.  The dimension is essential in the known gradient argument.
Gleason and Wolff developed a higher-dimensional theory, isolated an eigenvalue
form of the maximum principle behind the three-dimensional result, and obtained
conclusions under additional spectral assumptions \cite{GleasonWolff}.  For
general vector-valued harmonic homeomorphisms, the planar Lewy conclusion already
fails in higher dimension, as Wood showed \cite{Wood}.  Thus the gradient
structure is more rigid than an arbitrary harmonic mapping, but its direct
nondegeneracy property remains strongly dimension dependent.

The corresponding nonlinear question is whether local injectivity of the
gradient of a solution to a quasilinear equation forces Hessian nondegeneracy.
A recent result of
Li, Liu and Ma proves a local Lewy theorem for $p$-harmonic functions in $\R^3$
away from the critical set \cite{LiLiuMa}.  Their proof divides a singular Hessian
according to its rank.  The rank-two case is treated through a partial Legendre
transform, while the rank-zero case combines analytic blow-up with the
three-dimensional harmonic-gradient theory and a zero-Hessian theorem for ternary
forms.  Partial Legendre transforms have a long history in degenerate
Monge--Amp\`ere equations; useful modern treatments include
\cite{Guan1997,GuanSawyer,RiosSawyerWheeden}.

The passage from the $p$-Laplace coefficient to an arbitrary positive definite
field $a(Du)$ is not formal.  In the rank-two case, the transformed equation
must be reorganized so that ellipticity enters only through the congruence
factorization \eqref{eq:intro-symbol} below, while in the rank-zero case strong
unique continuation replaces the analytic expansion used in \cite{LiLiuMa}.

The first purpose of this paper is to identify a natural gradient-dependent
equation class for which that rank argument works.  We consider
\begin{equation}\label{eq:main}
 a^{ij}(Du)u_{ij}=0
 \qquad\text{in }\Omega\subset\R^n,
\end{equation}
where $a(q)$ is a smooth symmetric positive definite matrix field.  Unless
otherwise stated, repeated indices are summed from $1$ to $n$.  Let $\Sym_n$
denote the space of all $n\times n$ real symmetric matrices.  For
$a\in\Sym_n$, the notation $a>0$ means that $a$ is positive definite.
In H\"older spaces, the exponent $\alpha$ is understood to satisfy $0<\alpha<1$.
Our first main theorem is the following.

\begin{theorem}[Local Lewy theorem]\label{thm:local-intro}
Let $\mathcal U\subset\R^3$ be open and let
$a\in C^\infty(\mathcal U,\Sym_3)$ satisfy $a(q)>0$ for every
$q\in\mathcal U$.  Let $\Omega\subset\R^3$ be a domain and let
$u\in C^{2,\alpha}_{\mathrm{loc}}(\Omega)$ solve \eqref{eq:main} for $n=3$, with
$Du(\Omega)\subset\mathcal U$.  If
\[
 Du:\Omega\longrightarrow Du(\Omega)
\]
is a local homeomorphism, then
\[
 \det D^2u(x)\ne0\qquad\text{for every }x\in\Omega.
\]
Consequently $Du$ is a local $C^\infty$-diffeomorphism.  If $Du$ is globally
one-to-one onto its image, then it is a global $C^\infty$-diffeomorphism onto its
image.
\end{theorem}

There are four points in \cref{thm:local-intro} worth emphasizing.  First, the
coefficient matrix need not be rotationally invariant and need not arise from a
variational integral.  Second, positive definiteness is only local on gradient
space; no global ellipticity constants are required.  Third, smoothness suffices
also in the rank-zero alternative: strong unique continuation excludes
infinite-order contact with the affine tangent function and hence supplies a
first nonzero homogeneous Taylor term.  With merely $a\in C^2$ and $u\in C^4$,
the rank-one and rank-two arguments already show that every singular Hessian
must have rank zero.  Fourth,
the hypothesis that the coefficients depend only on $Du$ is structural.
After a partial Legendre transform $\phi$, the terms arising from the first
and second $s$-derivatives of $a(y,-\phi_s)$ are controlled by
$C(|w|+|Dw|)$, where $w=\phi_{ss}$.

The decisive new observation in the rank-two calculation is that the principal
symbol of the equation for the degenerate direction can be written as
\begin{equation}\label{eq:intro-symbol}
 \cQ(\xi,\tau)
 =\bigl(JB\xi+\tau e_3\bigr)^T
 a(y,-\phi_s)
 \bigl(JB\xi+\tau e_3\bigr).
\end{equation}
Here $B$ is the invertible tangential Hessian block, $\phi$ is the partial
Legendre transform, and $J$ is the natural graph matrix of the transform.  Thus
ellipticity is inherited by a congruence and does not depend on the special
formula of the $p$-Laplace coefficient.  The remaining cubic matrix term is
controlled by a genuinely two-dimensional cancellation.

The theorem has a particularly natural variational formulation.  If
$F\in C^\infty(\mathcal U)$ satisfies
\[
 D_q^2F(q)>0\qquad\text{for every }q\in\mathcal U,
\]
then
\begin{equation}\label{eq:variational}
 \operatorname{div}\bigl(D_qF(Du)\bigr)=0
\end{equation}
has the positive definite coefficient matrix $D_q^2F(Du)$.  We obtain the
following consequence.

\begin{corollary}\label{cor:variational-intro}
Let $\mathcal U\subset\R^3$ be open, let $F\in C^\infty(\mathcal U)$ satisfy
$D_q^2F(q)>0$ on $\mathcal U$, and let $\Omega\subset\R^3$ be a domain.  Suppose
that $u\in C^{2,\alpha}_{\mathrm{loc}}(\Omega)$ solves
\eqref{eq:variational}, that $Du(\Omega)\subset\mathcal U$, and that
$Du:\Omega\to Du(\Omega)$ is a local homeomorphism.  Then
$\det D^2u$ is nowhere zero in $\Omega$, and $Du$ is a local
$C^\infty$-diffeomorphism.
\end{corollary}

The choices
\[
 F(q)=\frac1p|q|^p,
 \qquad
 F(q)=\sqrt{1+|q|^2}
\]
give, respectively, the $p$-Laplace equation away from $q=0$ and the minimal
surface equation on all of gradient space.  Accordingly, we recover the recent
$p$-harmonic theorem and obtain its minimal-surface counterpart without a
nonvanishing-gradient assumption.  In both cases the coefficient matrix is real
analytic on the relevant gradient range, so analytic elliptic regularity and the
analytic inverse function theorem show that the local diffeomorphism is
real analytic.

More generally, let $\mathcal F:\mathcal U\to\R^3$ be smooth and consider
\[
\operatorname{div}\mathcal F(Du)=0.
\]
Because $D^2u$ is symmetric, this equation is
\[
b^{ij}u_{ij}=0,
\]
where
\[
 b=\frac{D\mathcal F(Du)+D\mathcal F(Du)^T}{2}.
\]
We therefore have the following corollary.

\begin{corollary}\label{cor:flux}
Let $\mathcal U\subset\R^3$ be open and let
$\mathcal F\in C^\infty(\mathcal U,\R^3)$ satisfy
\[
 \frac{D\mathcal F(q)+D\mathcal F(q)^T}{2}>0
 \qquad\text{for every }q\in\mathcal U.
\]
Let $\Omega\subset\R^3$ be a domain, and suppose that
$u\in C^{2,\alpha}_{\mathrm{loc}}(\Omega)$ solves
$\operatorname{div}\mathcal F(Du)=0$, that
$Du(\Omega)\subset\mathcal U$, and that
$Du:\Omega\to Du(\Omega)$ is a local homeomorphism.  Then
$\det D^2u$ is nowhere zero in $\Omega$, and $Du$ is a local
$C^\infty$-diffeomorphism.
\end{corollary}

Our second main result concerns a geometrically different regime.  Let
\[
 \Omega=\Omega_0\setminus\overline{\Omega_1}\subset\R^n
\]
be a convex ring, and consider a capacitary solution with constant values on the
two boundary components.  Throughout this paper, a smooth convex hypersurface
is called strictly convex if its second fundamental form is positive definite.
We use the sign convention for which spheres have positive principal curvatures
with respect to the outward unit normal.  The classical study of convex level sets in rings has
two complementary strands.  The macroscopic convexity maximum principles of
Gabriel, Lewis and Caffarelli--Spruck establish quasiconcavity for important
capacitary equations; see \cite{Gabriel,Lewis,CaffarelliSpruck} and Kawohl's
monograph \cite{Kawohl}.  The microscopic method begins with
Korevaar's constant-rank theorem for the second fundamental form of the level
sets \cite{Korevaar}.  It was developed further through quasiconcave-envelope,
constant-rank and curvature-estimate methods
\cite{BianGuanMaXu,BianchiniLonginettiSalani,GuanXu,MaOuZhang,WangZhang,ZhangZhang}.
In particular, for the radial two-eigenvalue operator
\[
 A(|Du|)(\Delta u-u_{\nu\nu})+B(|Du|)u_{\nu\nu}=0,
 \qquad \nu=\frac{Du}{|Du|},
\]
Korevaar's condition
\[
 \frac{d^2}{d\mu^2}\sqrt{\frac{A(\mu)}{B(\mu)}}\ge0
\]
propagates strict convexity by a constant-rank argument once convexity and
noncriticality are known and strict convexity holds at one point.  The
$p$-Laplace and minimal surface operators both
belong to this class.  Once strict level-set convexity is available, the Lewy
nondegeneracy becomes both stronger and simpler: no topological hypothesis on
$Du$ is needed, no analyticity is needed, and the result holds in every
dimension.

\begin{theorem}[Convex-ring gradient theorem]\label{thm:ring-intro}
Let $n\ge2$ and let
$\Omega=\Omega_0\setminus\overline{\Omega_1}\subset\R^n$, where
$\Omega_0$ and $\Omega_1$ are bounded smooth strictly convex domains and
$\overline{\Omega_1}\Subset\Omega_0$.  Let $\mathcal U\subset\R^n$ be open,
let $a\in C^0(\mathcal U,\Sym_n)$ satisfy $a(q)>0$ on $\mathcal U$, and let
$u\in C^3(\Omega)\cap C^2(\overline\Omega)$ solve
\begin{equation}\label{eq:ring-intro}
 \begin{cases}
 a^{ij}(Du)u_{ij}=0&\text{in }\Omega,\\
 u=0&\text{on }\partial\Omega_0,\\
 u=1&\text{on }\partial\Omega_1.
 \end{cases}
\end{equation}
Assume that $Du(\overline\Omega)\subset\mathcal U$, that $|Du|>0$ in
$\overline\Omega$, and that every superlevel body
\[
 K_t=\overline{\Omega_1}\cup\{x\in\overline\Omega:u(x)\ge t\},
 \qquad 0\le t\le1,
\]
has smooth strictly convex boundary.  Then
\[
 (-1)^{n-1}\det D^2u>0\qquad\text{in }\Omega.
\]
Moreover, $Du$ is a global $C^2$-diffeomorphism onto a radially parametrized
ring in gradient space.
\end{theorem}

The proof rests on an exact completed-square identity.  In an adapted frame the
tangential Hessian is negative definite.  Writing the Hessian and coefficient
matrix in blocks, the equation implies that the Schur complement of the
tangential Hessian is strictly positive.  This fixes the inertia of $D^2u$ as
$(1,n-1)$.  Parametrizing the level sets by their Gauss maps then shows that the
gradient magnitude is strictly increasing with the level value at fixed normal,
which gives global injectivity.

The two theorems are complementary rather than redundant.  The local theorem
requires no level-set convexity but is three-dimensional and smooth.  The ring
theorem works in every dimension with merely continuous coefficient fields,
once the stated geometric hypothesis of strict level-set convexity is
available.  For several classical equations, this
condition is itself a theorem rather than an imposed assumption.  For example,
Lewis' extension of Gabriel's theorem gives it for $p$-capacitary potentials
even when the boundary components are merely convex, while Korevaar's
deformation and constant-rank theory gives it for a broad radial quasilinear
class.  More general sufficient conditions for
quasiconcavity and constant rank are known for nonlinear elliptic operators
\cite{BianGuanMaXu,BianchiniLonginettiSalani,GuanXu}.  Thus the open issue is not
whether any structural theory exists, but how to formulate sharp, directly
verifiable, and affine-invariant conditions for a genuinely anisotropic matrix
field $a(q)$.  The role of our ring theorem is complementary to that issue:
every
such convexity criterion immediately yields Hessian nondegeneracy and a global
gradient diffeomorphism.

Their common conclusion is that, for this gradient-dependent coefficient
class, either local injectivity of the gradient in dimension three or strict
level-set convexity in arbitrary dimension rules out Hessian degeneracy.

The paper is organized as follows.  \Cref{sec:prelim} collects the topological,
elliptic-regularity, unique-continuation, and algebraic ingredients.
\Cref{sec:rankone} proves the local Lewy theorem by excluding Hessian ranks two,
one, and zero.  \Cref{sec:ring} proves the
arbitrary-dimensional convex-ring
theorem and explains its relation to existing level-set constant-rank
criteria.  \Cref{sec:examples} treats the $p$-Laplace and minimal surface
equations.

\section{Preliminaries}\label{sec:prelim}

The equation \eqref{eq:main} is invariant under invertible linear changes in the
domain, provided the coefficient matrix is transformed accordingly.  In
particular, orthogonal changes may be used to place a Hessian kernel along a
coordinate axis.  Assuming symmetry of $a$ entails no loss of generality for
the equation, since only
the symmetric part of a coefficient matrix contracts with $D^2u$.

All arguments below are local.  Thus the pointwise positive definiteness of $a$
supplies uniform ellipticity after the coordinate neighborhood and the
corresponding gradient range have been shrunk.  If $a\in C^\infty$ and
$u\in C^{2,\alpha}_{\mathrm{loc}}$ solves \eqref{eq:main}, standard interior
Schauder bootstrapping gives
\[
 u\in C^\infty(\Omega).
\]
Indeed, $a(Du)\in C^{1,\alpha}$ initially, so the linear Schauder estimates give
$u\in C^{3,\alpha}$; iteration proves the assertion.  When $a$ is real analytic,
analytic elliptic regularity further gives $u\in C^\omega$; see \cite{Morrey}.

We shall use the following standard form of the strong unique continuation
principle for second-order elliptic equations with Lipschitz leading
coefficients and bounded lower-order terms; see \cite{GarofaloLin}.

\begin{theorem}[Strong unique continuation principle]
\label{strong unique-continuation}
Let $\Omega\subset\R^n$, $n\ge2$, be a connected domain.  Suppose that
$A\in W^{1,\infty}_{\mathrm{loc}}(\Omega,\Sym_n)$ is locally uniformly
positive definite and that $\mathbf b,V\in L^\infty_{\mathrm{loc}}(\Omega)$.
If $v\in H^1_{\mathrm{loc}}(\Omega)$ is a weak solution of
\[
 -\operatorname{div}(A(x)Dv)+\mathbf b(x)\cdot Dv+V(x)v=0
 \qquad\text{in }\Omega
\]
and vanishes to infinite order at some $x_0\in\Omega$, in the sense that
\[
 \int_{B_r(x_0)}v^2\,dx=O(r^N)
 \qquad\text{as }r\downarrow0\quad\text{for every }N>0,
\]
then $v\equiv0$ in $\Omega$.
\end{theorem}

For a smooth function, vanishing of every Taylor coefficient implies the
required infinite-order $L^2$ vanishing.

\begin{lemma}\label{lem:orientation}
Let $V\subset\R^3$ be connected and let $G\in C^1(V,\R^3)$ be a local
homeomorphism.  Then either
\[
 \det DG\ge0\quad\text{in }V,
\]
or
\[
 \det DG\le0\quad\text{in }V.
\]
\end{lemma}
\begin{proof}
A local homeomorphism between oriented connected manifolds has a fixed
orientation character.  Equivalently, its local degree is everywhere $+1$ or
everywhere $-1$; see \cite{Hatcher}.  At a regular point this degree equals the
sign of the Jacobian.

It remains to note that regular points are dense.  If $\det DG$ vanished on an
open ball $B\Subset V$, the area formula would give
\[
 |G(B)|\le\int_B|\det DG|\dd x=0;
\]
see \cite{EvansGariepy}.  This contradicts invariance of domain, since a local
homeomorphism is open.  Hence every open ball contains a regular point, and
continuity gives the fixed weak sign.
\end{proof}

We use the following form of the Gleason--Wolff maximum principle
\cite{GleasonWolff}.

\begin{theorem}[Gleason--Wolff]\label{thm:GW}
Let $Q$ be harmonic in a connected domain in $\R^3$.  Suppose that $D^2Q$ has at
most one negative eigenvalue at every point.  If
\[
 \det D^2Q(x_0)=0
\]
at one point, then
\[
 \det D^2Q\equiv0.
\]
\end{theorem}

For later use, if $M\in\Sym_3$ and $\tr M=0$, then
\[
 \det M\le0
 \quad\Longrightarrow\quad
 M\text{ has at most one negative eigenvalue}.
\autotag\label{eq:trace-sign}
\]
Indeed, two negative eigenvalues and one positive eigenvalue would give positive
determinant, while three negative eigenvalues are incompatible with zero trace.
If $\det M\ge0$, the same observation applies to $-M$.

We also use the ternary case of the Gordan--Noether theorem.  The classical
source is \cite{GordanNoether}; a modern detailed account, including corrections
of subtleties in general dimension, is \cite{WatanabeBondt}.

\begin{theorem}[Ternary Gordan--Noether theorem]\label{thm:GN}
Let $Q$ be a real homogeneous polynomial in three variables, of degree at least
three.  If
\[
 \det D^2Q\equiv0,
\]
then, after an invertible real linear change of variables,
\[
 Q(X)=Q_0(X_1,X_2).
\]
\end{theorem}

The real statement follows from the complex theorem.  If a complex vector
$\mathbf{z}=\mathbf{x}+i\mathbf{y}\ne0$ satisfies $\mathbf{z}\cdot DQ=0$, then, because $Q$ has real coefficients,
$\mathbf{x}\cdot DQ=\mathbf{y}\cdot DQ=0$.  At least one of $\mathbf{x}$ and $\mathbf{y}$ is a nonzero real
direction in
which $Q$ is constant.

The rank-zero argument ends with a purely topological obstruction.  The
following lemma explains why a planar complex power cannot be the leading-order
part of an injective three-dimensional mapping.  A proof may be found in
\cite[Lemma~2.4]{LiLiuMa}; for completeness and to keep the present argument
self-contained, we include it here.

\begin{lemma}\label{lem:monodromy}
Let $m\ge2$ and let $G=(G_1,G_2,G_3)$ be continuous in a neighborhood of
$0\in\R^3$.  Write
\[
 \zeta=\xi+i\eta,\qquad t=x_3,
 \qquad P=G_1+iG_2,\qquad Z=G_3.
\]
Assume that $P$ is $C^1$ and, after invertible real linear changes in the domain
and target,
\begin{equation}\label{eq:mon-leading}
 P(\zeta,t)=c\zeta^m+R(\zeta,t),\qquad c\ne0,
\end{equation}
where
\begin{equation}\label{eq:mon-error}
 |R(x)|=O(|x|^{m+1}),
 \qquad |DR(x)|=O(|x|^m).
\end{equation}
Then $G$ is not one-to-one in any neighborhood of the origin.
\end{lemma}

\begin{proof}
After an invertible real linear change in the first two target coordinates, we
may assume that $c=1$.  Suppose, toward a contradiction, that $G$ is one-to-one
in a ball centered at the origin.  Fix
\[
 0<\eps<\frac{\pi}{4m}.
\]
For sufficiently small $r>0$, write
\[
 \zeta=r\rho e^{i\theta},\qquad t=r\tau,
\]
where $|\rho-1|\le\eps$ and $|\tau|\le\eps$.  By
\eqref{eq:mon-leading} and \eqref{eq:mon-error},
\[
 r^{-m}P(r\rho e^{i\theta},r\tau)
 =\rho^me^{im\theta}+E_r(\rho,\theta,\tau),
\]
where $E_r\to0$ in $C^1$, uniformly on this compact parameter set as
$r\downarrow0$.

For $j=0,\ldots,m-1$, set
\[
 \theta=\frac{\varphi+2\pi j}{m}+\alpha.
\]
The equation
\begin{equation}\label{eq:mon-target}
 P(r\rho e^{i\theta},r\tau)=r^me^{i\varphi}
\end{equation}
reduces, after multiplication by $e^{-i\varphi}$, to a $C^1$-small
perturbation of
\[
 g(\rho,\alpha):=\rho^me^{im\alpha}-1=0.
\]
On the rectangle
\[
 K=\{(\rho,\alpha):|\rho-1|\le\eps,
                     \ |\alpha|\le\eps\},
\]
the equation $g=0$ has the unique solution $(1,0)$, and
$D_{(\rho,\alpha)}g(1,0)$ is an isomorphism from $\R^2$ to
$\mathbb C\simeq\R^2$.  Choose $\delta>0$ such that
$\overline{B_\delta(1,0)}\subset\operatorname{int}K$ and the real Jacobian of
$g$ is positive on $\overline{B_\delta(1,0)}$.  Since $g$ has no zero on
$K\setminus B_\delta(1,0)$, it is bounded away from zero there.  Uniform
$C^1$ convergence now implies, for every
$(\varphi,\tau,j)\in[0,2\pi]\times[-\eps,\eps]\times\{0,\ldots,m-1\}$ and all
sufficiently small $r$, that the perturbed equation has no zero in
$K\setminus B_\delta(1,0)$ and has positive Jacobian in
$\overline{B_\delta(1,0)}$.  For sufficiently small $r$, the straight-line
homotopy between $g$ and the perturbed map has no zero on
$\partial B_\delta(1,0)$.  Homotopy invariance of the Brouwer degree therefore
gives degree one on this ball.  Each zero contributes local degree $+1$,
so the perturbed equation has exactly one zero in $K$.

Denote this zero by
\[
 (\rho_j(\varphi,\tau),\alpha_j(\varphi,\tau)).
\]
The implicit function theorem with parameters shows that these functions are
$C^1$ in $(\varphi,\tau)\in[0,2\pi]\times[-\eps,\eps]$.  Define
\[
 \theta_j(\varphi,\tau)
 =\frac{\varphi+2\pi j}{m}+\alpha_j(\varphi,\tau)
\]
and
\[
 x_j^\varphi(\tau)=
 \left(
 r\rho_j(\varphi,\tau)\cos\theta_j(\varphi,\tau),
 r\rho_j(\varphi,\tau)\sin\theta_j(\varphi,\tau),
 r\tau
 \right).
\]
Then
\[
 P(x_j^\varphi(\tau))=r^me^{i\varphi}
\]
for every $j$ and $\tau$.  Uniqueness on the lifted sheets gives the cyclic
monodromy
\begin{equation}\label{eq:mon-cycle}
 x_j^{2\pi}(\tau)=x_{j+1}^{0}(\tau),
\end{equation}
where indices are taken modulo $m$.

For fixed $\varphi$, set
\[
 I_j(\varphi)
 =Z\bigl(x_j^\varphi([-\eps,\eps])\bigr).
\]
Each $I_j(\varphi)$ is a compact interval.  Injectivity of $G$ implies that
$\tau\mapsto Z(x_j^\varphi(\tau))$ is one-to-one.  Indeed, equality of two
$Z$-values, together with the common $P$-value, would give equality of the
corresponding $G$-values; equality of the points would then force equality of
their third coordinates and hence of the two parameters.

The intervals $I_j(\varphi)$ are pairwise disjoint.  Otherwise, two points on
distinct lifted sheets would have the same $P$- and $Z$-values.  Injectivity of
$G$ would make the points equal, which is impossible because their angular
separation is at least $2\pi/m-2\eps>0$ modulo $2\pi$.

Because $Z$ and the parametrized sheets are continuous, the two endpoints of
each interval $I_j(\varphi)$ depend continuously on $\varphi$.  Hence the
left-to-right order of the labelled pairwise disjoint intervals
\[
 I_0(\varphi),\ldots,I_{m-1}(\varphi)
\]
is locally constant and therefore constant for $0\le\varphi\le2\pi$.  On the
other hand, \eqref{eq:mon-cycle} gives
\[
 I_j(2\pi)=I_{j+1}(0).
\]
A nontrivial cyclic permutation of $m\ge2$ objects cannot preserve a strict
linear order.  This contradiction proves the lemma.
\end{proof}

\section{The local Lewy theorem in dimension three}
\label{sec:rankone}
We first exclude the rank-two case.
\begin{proposition}\label{prop:ranktwo}
	Let $\mathcal U\subset\R^3$ be open, let
	$a\in C^2(\mathcal U,\Sym_3)$ satisfy $a(q)>0$ on $\mathcal U$, and let
	$u\in C^4(\Omega)$ solve $a^{ij}(Du)u_{ij}=0$, with
	$Du(\Omega)\subset\mathcal U$.  If $Du$ is a local homeomorphism, then no
	point $x_0\in\Omega$ can satisfy
	\[
	\rank D^2u(x_0)=2.
	\]
\end{proposition}

\begin{proof}
Suppose that
\begin{align*}
 x_0=0,
 \qquad \rank D^2u(0)=2.
\end{align*}
After an orthogonal change of variables, take
\[
 \ker D^2u(0)=\operatorname{span}\{e_3\}.
\]
Write $x=(x',s)$, with $x'=(x_1,x_2)$ and $s=x_3$.  The block
\[
 B(0)=D^2_{x'x'}u(0)
\]
is invertible.  Hence
\[
 (x',s)\longmapsto(y,s):=(D_{x'}u(x',s),s)
\]
is a local diffeomorphism.  Define
\[
 \phi(y,s)=x'(y,s)\cdot y-u(x'(y,s),s).
\]
Then
\[
 D_y\phi=x',
 \qquad \phi_s=-u_s,
 \qquad Du=(y,-\phi_s).
\]
Set
\[
 P=D_y^2\phi,
 \qquad B=D^2_{x'x'}u=P^{-1},
 \qquad X=x'_s=D_y\phi_s,
 \qquad w=\phi_{ss}.
\]
Differentiating both sides of $y=D_{x'}u(x',s)$ with respect to
$s$, we obtain
\[
0=D^2_{x'x'}ux_s^{'}+D_{x'}u_s,
\]
that is,
\[
 D_{x'}u_s=-BX.
\]
Consequently
\[
 D^2u=
 \begin{pmatrix}
 B&-BX\\
 -X^TB&X^TBX-w
 \end{pmatrix}.
\autotag\label{eq:hessian-block}
\]
Set
\[
 J=\begin{pmatrix}I_2\\-X^T\end{pmatrix}.
\autotag\label{eq:J-def}
\]
Then
\[
 D^2u=JBJ^T-w e_3e_3^T.
\autotag\label{eq:hessian-factor}
\]
Let
\[
 \cA=a(y,-\phi_s),
 \qquad M=J^T\cA J,
 \qquad c=e_3^T\cA e_3.
\autotag\label{eq:Mc-def}
\]
The matrix $M$ is positive definite, and $c>0$.  From
\eqref{eq:hessian-factor} and \eqref{eq:Mc-def}, equation \eqref{eq:main}
becomes
\begin{equation}\label{eq:transformed}
\tr(BM)=cw.
\end{equation}

Define
\[
 G(y,s)=(y,-\phi_s(y,s)).
\]
If $H(x)=(D_{x'}u(x),x_3)$, then $H$ is a local diffeomorphism and $Du=G\circ H$.  Thus $G=Du\circ H^{-1}$ is a local
homeomorphism.  Moreover,
\[
 DG=
 \begin{pmatrix}
 I_2&0\\
 -(D_y\phi_s)^T&-\phi_{ss}
 \end{pmatrix},
 \qquad \det DG=-w.
\]
By \cref{lem:orientation}, after shrinking the coordinate neighborhood there is
$\sigma\in\{-1,1\}$ such that
\[
 \sigma w\ge0.
\autotag\label{eq:w-sign}
\]
From $D^2u(0)e_3=0$ and \eqref{eq:hessian-block}, the invertibility of $B(0)$
gives
\[
 X(0)=0,
 \qquad w(0)=0.
\]
Thus $\sigma w$ has an interior minimum at the origin and $Dw(0)=0$.  We next
prove a differential inequality for $w$.  Because $BP=I$,
\[
 B_s=-BP_sB,
 \qquad
 B_{ss}=2BP_sBP_sB-BD_y^2wB.
\autotag\label{eq:B-derivatives}
\]
The variables in $M$ are
\[
 M=M(y,z,X),
 \qquad z=\phi_s,
 \qquad X=D_y\phi_s.
\]
At fixed $y$,
\[
 z_s=w,
 \qquad X_s=D_yw,
\]
and therefore
\[
 M_s=M_zw+M_{X_\alpha}w_\alpha.
\autotag\label{eq:Ms}
\]
Similarly, $c=c(y,z)$ and $c_s=c_zw$. Differentiating \eqref{eq:transformed} once gives
\[
 -\tr(BP_sBM)+\tr(BM_s)-c_sw-cw_s=0.
\]
After shrinking the neighborhood, all coefficients and the derivatives of
$\phi$ which occur here are bounded.  Hence
\begin{align}
 \tr(BM)&=O(|w|),\label{eq:traceS}\autotag\\
 \tr(BP_sBM)&=O(|w|+|Dw|).\label{eq:traceSNS}\autotag
\end{align}
Differentiating a second time and using \eqref{eq:B-derivatives},
\begin{align}
0={}&2\tr(BP_sBP_sBM)-\tr(BD_y^2wBM)
 -2\tr(BP_sBM_s)\notag\\
&+\tr(BM_{ss})-c_{ss}w-2c_sw_s-cw_{ss}.
\label{eq:second-diff}\autotag
\end{align}
The only apparently uncontrolled term in \eqref{eq:second-diff} is
\[
 \mathcal C=\tr(BP_sBP_sBM).
\]

\begin{lemma}\label{lem:cubic}
In a sufficiently small neighborhood of the origin,
\[
 |\mathcal C|\le C(|w|+|Dw|).
\]
\end{lemma}

\begin{proof}
Set
\[
 S=M^{1/2}BM^{1/2},
 \qquad N=M^{-1/2}P_sM^{-1/2}.
\]
Then
\[
 \tr(BM)=\tr S,
 \qquad \tr(BP_sBM)=\tr(SNS),
 \qquad \mathcal C=\tr(SNSNS).
\]
At the origin, $w=0$ and \eqref{eq:transformed} gives $\tr(BM)=0$.  Since $M>0$
and $B$ is invertible, $B(0)$ has one positive and one negative eigenvalue.
Thus the eigenvalues $\lambda,\mu$ of $S$ remain of opposite signs and bounded
away from zero in a small neighborhood. Diagonalize $S$ and write
\[
 S=\begin{pmatrix}\lambda&0\\0&\mu\end{pmatrix},
 \qquad
 N=\begin{pmatrix}\alpha&\beta\\\beta&\gamma\end{pmatrix}.
\]
Equations \eqref{eq:traceS}--\eqref{eq:traceSNS} become
\[
 \lambda+\mu=O(|w|),
 \qquad
 \lambda^2\alpha+\mu^2\gamma=O(|w|+|Dw|).
\]
Because $\mu=-\lambda+O(|w|)$ and $|\lambda|$ is bounded away from zero,
\[
 \alpha+\gamma=O(|w|+|Dw|).
\autotag\label{eq:alpha-gamma}
\]
Finally,
\[
 \tr(SNSNS)
 =\lambda^3\alpha^2+\mu^3\gamma^2
  +\lambda\mu(\lambda+\mu)\beta^2.
\]
The last term is $O(|w|)$.  For the first two terms,
\[
 \lambda^3\alpha^2+\mu^3\gamma^2
 =\lambda^3(\alpha^2-\gamma^2)+O(|w|)
 =\lambda^3(\alpha-\gamma)(\alpha+\gamma)+O(|w|).
\]
Since $N$ is bounded, \eqref{eq:alpha-gamma} completes the proof.
\end{proof}

 After shrinking once more, assume $|w|+|Dw|\le1$.  From
\eqref{eq:Ms},
\[
 M_{ss}=M_{X_\alpha}w_{\alpha s}
 +O(|w|+|Dw|).
\]
The terms $\tr(BP_sBM_s)$, $c_{ss}w$, and $c_sw_s$ are also
$O(|w|+|Dw|)$.  Define
\[
 r=J^T\cA e_3\in\R^2.
\autotag\label{eq:r-def}
\]
Because $\partial_{X_\alpha}J=-e_3e_\alpha^T$,
\[
 M_{X_\alpha}=-e_\alpha r^T-r e_\alpha^T,
 \qquad
 \tr(BM_{X_\alpha})=-2(Br)_\alpha.
\]
Using \cref{lem:cubic} in \eqref{eq:second-diff}, we obtain
\begin{equation}\label{eq:Lw-ineq}
 |Lw|\le C(|w|+|Dw|),
\end{equation}
where
\begin{equation}\label{eq:L-def}
 Lw=\tr(BD_y^2wBM)+2Br\cdot D_yw_s+cw_{ss}.
\end{equation}

\begin{lemma}\label{lem:L-elliptic}
The operator $L$ in \eqref{eq:L-def} is uniformly elliptic in a sufficiently
small neighborhood of the origin.
\end{lemma}

\begin{proof}
Its principal symbol at $(\xi,\tau)\in\R^2\times\R$ is
\[
 \cQ(\xi,\tau)
 =\xi^TBMB\xi+2\tau Br\cdot\xi+c\tau^2.
\]
Using \eqref{eq:Mc-def}, \eqref{eq:J-def} and \eqref{eq:r-def}, we obtain the
exact factorization
\[
 \cQ(\xi,\tau)
 =\bigl(JB\xi+\tau e_3\bigr)^T\cA
  \bigl(JB\xi+\tau e_3\bigr).
\autotag\label{eq:symbol-factorization}
\]
The linear map $(\xi,\tau)\mapsto JB\xi+\tau e_3$ has matrix
\[
 \begin{pmatrix}
 B&0\\
 -X^TB&1
 \end{pmatrix},
\]
whose determinant is $\det B\ne0$.  Its inverse is locally uniformly bounded.
Positive definiteness of $\cA$ now gives
\[
 \lambda_0(|\xi|^2+\tau^2)
 \le\cQ(\xi,\tau)
 \le\Lambda_0(|\xi|^2+\tau^2)
\]
for some $0<\lambda_0\le\Lambda_0$.
\end{proof}

Let $v=\sigma w$, where $\sigma$ is chosen in \eqref{eq:w-sign}.  Then
\[
 v\ge0,
 \qquad v(0)=0,
 \qquad Lv\le C(v+|Dv|).
\]
The resulting operator is uniformly elliptic, has bounded drift, and has
nonpositive zeroth-order coefficient.  The strong minimum principle
\cite[Theorem~3.5]{GilbargTrudinger} implies $v\equiv0$ in a neighborhood of
the origin.

Thus $w=\phi_{ss}\equiv0$.  For fixed $y$, $\phi_s(y,s)$ is independent of $s$,
so $G(y,s)=(y,-\phi_s(y,s))$ takes the same value at distinct nearby $s$.
This contradicts the local injectivity of $G$.
\end{proof}

Next, we rule out the rank-one case.
\begin{proposition}\label{prop:rankone}
	Let $\mathcal U\subset\R^3$ be open, let
	$a\in C^0(\mathcal U,\Sym_3)$ satisfy $a(q)>0$ on $\mathcal U$, and let
	$u\in C^2(\Omega)$ solve $a^{ij}(Du)u_{ij}=0$, with
	$Du(\Omega)\subset\mathcal U$.  Then no point $x_0\in\Omega$ can satisfy
	\[
	\rank D^2u(x_0)=1.
	\]
\end{proposition}

\begin{proof}
	If the rank is one, then
	\[
	D^2u(x_0)=\lambda\,\xi\otimes\xi
	\]
	for some $\lambda\ne0$ and $\xi\ne0$.  Evaluating the equation at $x_0$ gives
	\[
	0=\lambda\,\xi^Ta(Du(x_0))\xi.
	\]
	The last quadratic form is strictly positive, which forces $\lambda=0$, a
	contradiction.
\end{proof}

Finally, we exclude the rank-zero case.
\begin{proposition}\label{prop:rankzero}
	Let $\mathcal U\subset\R^3$ be open, let
	$a\in C^\infty(\mathcal U,\Sym_3)$ satisfy $a(q)>0$ on $\mathcal U$, and let
	$u\in C^{2,\alpha}_{\mathrm{loc}}(\Omega)$ solve
	$a^{ij}(Du)u_{ij}=0$, with $Du(\Omega)\subset\mathcal U$.  If $Du$ is a
	local homeomorphism, then no point $x_0\in\Omega$ can satisfy
	\[
	\rank D^2u(x_0)=0.
	\]
\end{proposition}
\begin{proof}

Interior Schauder bootstrapping first gives $u\in C^\infty(\Omega)$. Suppose
\[
 D^2u(0)=0,
 \qquad q_0=Du(0).
\]
Let
\[
 \ell(x)=u(0)+q_0\cdot x,
 \qquad v=u-\ell.
\]
Since $D^2\ell=0$, the original equation itself becomes the linear equation
\begin{equation}\label{eq:v-linear}
 A^{ij}(x)v_{ij}=0,
 \qquad A(x)=a(Du(x)).
\end{equation}
The dependence of $A$ on the already fixed solution $u$ causes no difficulty:
on a sufficiently small ball, $A$ is a smooth, symmetric, uniformly positive
definite coefficient field.  In divergence form, \eqref{eq:v-linear} is
\begin{equation}\label{eq:v-divergence}
 \partial_i(A^{ij}\partial_jv)+b^j\partial_jv=0,
 \qquad b^j=-\partial_iA^{ij}.
\end{equation}
After multiplying by $-1$ if necessary, the strong unique continuation
principle recorded in \cref{sec:prelim} applies.  If every derivative of $v$
vanished at the origin, Taylor's theorem would give,
for every positive integer $N$,
\[
 |v(x)|=O(|x|^N).
\]
Hence $v$ would vanish to infinite order.  \Cref{strong unique-continuation}
would imply $v\equiv0$ in a neighborhood of the origin.  This would make
$Du\equiv q_0$ there, contradicting that $Du$ is a local homeomorphism.  It
follows that $v$ has finite order of vanishing.  Since
\[
 v(0)=0,\qquad Dv(0)=0,\qquad D^2v(0)=0,
\]
there is an integer $d\ge3$ and a nonzero homogeneous polynomial $P_d$ of
degree $d$ such that
\[
 u(x)=u(0)+q_0\cdot x+P_d(x)+R(x),
\]
where
\[
 R(x)=O(|x|^{d+1}),
 \quad DR(x)=O(|x|^d),
 \quad D^2R(x)=O(|x|^{d-1}).
\]
This is an ordinary finite-order Taylor expansion; no analyticity is used.  For
$x=r\theta$, homogeneity gives
\[
 Du(r\theta)=q_0+r^{d-1}DP_d(\theta)+O(r^d),
\]
\[
 D^2u(r\theta)=r^{d-2}D^2P_d(\theta)+O(r^{d-1}).
\]
Since $a$ is smooth,
\[
 a(Du(r\theta))=a(q_0)+O(r^{d-1}).
\]
Substituting into the equation,
\begin{align*}
0&=a^{ij}(Du(r\theta))u_{ij}(r\theta)\\
 &=\bigl(a^{ij}(q_0)+O(r^{d-1})\bigr)
   \bigl(r^{d-2}(P_d)_{ij}(\theta)+O(r^{d-1})\bigr)\\
 &=r^{d-2}a^{ij}(q_0)(P_d)_{ij}(\theta)+O(r^{d-1}).
\end{align*}
Dividing by $r^{d-2}$ and letting $r\to 0$, we get
\[
 a^{ij}(q_0)(P_d)_{ij}(\theta)=0.
\]
Set $A_0=a(q_0)$ and choose its positive square root $L=A_0^{1/2}$.  Define
\[
 Q(X)=P_d(LX).
\]
Then
\[
 \Delta Q(X)
 =\tr\bigl(L^TD^2P_d(LX)L\bigr)
 =A_0^{ij}(P_d)_{ij}(LX)=0.
\]
Moreover,
\[
 D^2Q(X)=L^TD^2P_d(LX)L,
\]
and hence
\[
 \det D^2Q(X)=(\det L)^2\det D^2P_d(LX).
\]
The two determinants have the same sign pattern.  By \cref{lem:orientation},
$\det D^2u$ has a fixed sign near the origin.
On the other hand,
\begin{equation}\label{eq:det-expansion}
 \det D^2u(r\theta)
 =r^{3(d-2)}\det D^2P_d(\theta)+O(r^{3d-5}).
\end{equation}
If $\det D^2P_d$ took both positive and negative values on the sphere, then
\eqref{eq:det-expansion} would force $\det D^2u$ to take both signs in every
neighborhood of the origin.  Therefore $\det D^2P_d$, and thus
$\det D^2Q$, has a fixed sign.  If $\det D^2Q\le0$, the trace-free matrix
$D^2Q$ has at most one negative
eigenvalue by \eqref{eq:trace-sign}.  Since $d\ge3$,
\[
 D^2Q(0)=0.
\]
The Gleason--Wolff theorem gives
\[
 \det D^2Q\equiv0.
\autotag\label{eq:Q-zero-Hessian}
\]
If $\det D^2Q\ge0$, apply the same argument to $-Q$.  Thus
\eqref{eq:Q-zero-Hessian} holds in either case.

By \cref{thm:GN}, after a real invertible linear change,
\[
 Q(X)=Q_0(X_1,X_2).
\]
Under this change of variables the Euclidean Laplacian becomes a
constant-coefficient positive definite operator.  Its restriction to functions
independent of $X_3$ is a positive definite two-dimensional operator.  A further
real linear change in $(X_1,X_2)$ reduces it to the planar Laplacian.  Since $Q_0$
is homogeneous and nonzero,
\[
 Q_0(X_1,X_2)
 =\Re\bigl(c(X_1+iX_2)^d\bigr)
\]
for some $c\ne0$.  Returning to the original variables only composes the
leading gradient with
invertible linear maps in the domain and target.  Thus, after such maps,
\[
 Du(x)-Du(0)
 =\left(
 \Re(c_1\zeta^{d-1}),
 -\Im(c_1\zeta^{d-1}),
 0
 \right)+O(|(\zeta,t)|^d),
\]
where $c_1\ne0$ and $\zeta=\xi+i\eta$.  The expansion is valid in $C^1$, so the
remainder in the first two components satisfies the derivative estimate in
\cref{lem:monodromy} with $m=d-1\ge2$.  That lemma contradicts the local
injectivity of $Du$.
\end{proof}

\begin{proof}[Proof of \cref{thm:local-intro}]
Combining \cref{prop:rankone,prop:ranktwo,prop:rankzero}, we have
$\det D^2u(x)\ne0$ for every $x\in\Omega$.  The inverse function theorem gives
the local-diffeomorphism
conclusion.  If $Du$ is globally injective onto its image, its local smooth
inverse branches agree, yielding a global diffeomorphism onto $Du(\Omega)$.
\end{proof}
\section{The convex-ring theorem in arbitrary dimension}
\label{sec:ring}

The local theorem above does not assume convex level sets.  Conversely, strict
level-set convexity gives a stronger result by elementary matrix geometry.  In
this section $n\ge2$, and analyticity is not assumed.  We first prove the
following Schur-complement lemma.

\begin{lemma}\label{lem:ring-algebra}
Let
\[
 H=\begin{pmatrix}-C&b\\b^T&d\end{pmatrix}\in\Sym_n,
 \qquad C\in\Sym_{n-1},\quad C>0,
\]
and let
\[
 A_*=\begin{pmatrix}A&\beta\\\beta^T&c\end{pmatrix}\in\Sym_n,
 \qquad A_*>0.
\]
If
\[
 A_*^{ij}H_{ij}=0,
\]
then the Schur complement
\[
 S=d+b^TC^{-1}b
\]
is positive.  More precisely,
\begin{align}
 cS={}&\tr\left[\left(A-\frac{\beta\beta^T}{c}\right)C\right]\notag\\
 &+c\left(b-\frac{C\beta}{c}\right)^TC^{-1}
      \left(b-\frac{C\beta}{c}\right)>0.
\label{eq:ring-square}\autotag
\end{align}
Consequently $H$ has inertia $(1,n-1)$ and
\[
 (-1)^{n-1}\det H>0.
\]
\end{lemma}

\begin{proof}
The equation $A_*^{ij}H_{ij}=0$ is equivalent to
\[
 -\tr(AC)+2\beta^Tb+cd=0.
\]
Hence
\[
 cS=\tr(AC)-2\beta^Tb+c b^TC^{-1}b.
\]
Completing the square gives \eqref{eq:ring-square}.  Since $A_*>0$, its Schur
complement satisfies
\[
 A-\frac{\beta\beta^T}{c}>0.
\]
The first term on the right of \eqref{eq:ring-square} is therefore strictly
positive, while the second is nonnegative.  Finally,
\[
 \det H=\det(-C)S=(-1)^{n-1}\det C\,S.
\]
The inertia statement follows from the Schur-complement congruence.
\end{proof}

\begin{proof}[Proof of \cref{thm:ring-intro}]
Let $u$ satisfy the hypotheses of \cref{thm:ring-intro}.
The maximum principle gives $0<u<1$ in $\Omega$.  Set
\[
 s=|Du|,
 \qquad \nu=\frac{Du}{|Du|}.
\]
For $0<t<1$, the vector $\nu$ points toward the interior of the superlevel body
$K_t$, while the outward unit normal of $K_t$ is $-\nu$.  Choose an orthonormal
frame $e_1,\ldots,e_{n-1},e_n=\nu$.  If $h$ denotes the positive
definite second fundamental form of $\partial K_t$ with respect to its outward
normal, then
\[
 u_{\alpha\beta}
 =-s h_{\alpha\beta},
 \qquad 1\le\alpha,\beta\le n-1.
\]
Thus the tangential Hessian is negative definite.  Write
\[
 D^2u=\begin{pmatrix}-C&b\\b^T&d\end{pmatrix},
 \qquad C>0,
\]
and decompose $a(Du)$ in the same frame.  The equation and
\cref{lem:ring-algebra} yield
\[
 (-1)^{n-1}\det D^2u>0,
\autotag\label{eq:ring-det}
\]
and
\[
 \operatorname{Inertia}(D^2u)=(1,n-1).
\autotag\label{eq:ring-inertia}
\]
 
 Strict convexity makes the Gauss map of each $\partial K_t$ a diffeomorphism.
Let $x=x(t,\nu)$ be the unique point of $\partial K_t$ whose inward normal is
$\nu\in\Sph^{n-1}$.  Define
\[
 Du(x(t,\nu))=s(t,\nu)\nu.
\]
At fixed $\nu$, differentiate with respect to $t$:
\[
 D^2u(x(t,\nu))x_t=s_t\nu.
\]
Since $u(x(t,\nu))=t$,
\[
 s\nu\cdot x_t=1.
\autotag\label{eq:height-diff}
\]
The Hessian is invertible by \eqref{eq:ring-det}; hence
\[
 x_t=s_t(D^2u)^{-1}\nu.
\]
In the adapted block decomposition, the $(n,n)$ entry of $(D^2u)^{-1}$ is
$S^{-1}$, where
\[
 S=d+b^TC^{-1}b>0.
\]
Substitution into \eqref{eq:height-diff} gives
\begin{equation}\label{eq:s-monotone}
s_t(t,\nu)=\frac{S(t,\nu)}{s(t,\nu)}>0.
\end{equation}

Suppose $Du(x_1)=Du(x_2)$.  The two points have the same normal $\nu$ and the
same gradient magnitude.  By \eqref{eq:s-monotone}, they have the same level
value $t$.  The Gauss map of that strictly convex level surface is injective, so
$x_1=x_2$.  Thus $Du$ is globally injective.  Since it is already a local
diffeomorphism, it is a global diffeomorphism onto its image.

The Hopf boundary point lemma also gives $|Du|>0$ on both boundary components.  To describe
the image, let $x_i(\nu)\in\partial\Omega_i$ be the point at which
$Du/|Du|=\nu$, and define
\[
 \rho_i(\nu)=|Du(x_i(\nu))|,
 \qquad i=0,1.
\]
Equation \eqref{eq:s-monotone} gives $\rho_0(\nu)<\rho_1(\nu)$ and
\begin{equation}\label{eq:gradient-ring}
 Du(\Omega)
 =\{r\nu:\nu\in\Sph^{n-1},\ \rho_0(\nu)<r<\rho_1(\nu)\}.
\end{equation}
This completes the proof of \cref{thm:ring-intro}.
\end{proof}
\begin{remark}\label{rem:ring-low-reg}
The pointwise determinant conclusion uses only continuity and positive
definiteness of $a$, twice differentiability of $u$, nonvanishing gradient, and
strict convexity of the level hypersurface through the point.  The additional
regularity in \cref{thm:ring-intro} is convenient for the global Gauss-map
parametrization.
\end{remark}


It is useful to record precisely how \cref{thm:ring-intro} relates to the
classical constant-rank theory.  For $q\ne0$, set $\hat q=q/|q|$ and consider
the radial two-eigenvalue field
\begin{equation}\label{eq:korevaar-field}
 a_{A,B}(q)
 =A(|q|)(I-\hat q\otimes\hat q)+B(|q|)\hat q\otimes\hat q,
 \qquad A,B>0.
\end{equation}
Thus $a_{A,B}^{ij}(Du)u_{ij}=0$ is
\[
 A(\mu)(\Delta u-u_{\nu\nu})+B(\mu)u_{\nu\nu}=0,
 \qquad \mu=|Du|,\quad \nu=\frac{Du}{|Du|}.
\]
Korevaar proved that, for noncritical solutions with convex level sets, the
second fundamental forms of the level sets have constant rank throughout the
connected ring provided
\begin{equation}\label{eq:korevaar-condition}
 \left(\sqrt{\frac{A}{B}}\right)''(\mu)\ge0.
\end{equation}
This is the homogeneous, gradient-only specialization of his more general
condition.  Condition \eqref{eq:korevaar-condition} is a constant-rank
criterion under the stated convexity and noncriticality hypotheses; the full
convex-ring theorem additionally uses barriers, comparison, and deformation to
establish those hypotheses.  The following corollary combines this geometric
result with our Schur-complement argument.

\begin{corollary}\label{cor:korevaar-ring}
Let $\Omega=\Omega_0\setminus\overline{\Omega_1}$ be a smooth strictly convex
ring in $\R^n$.  Let $A,B\in C^2(I)$ be positive on an open interval
$I\Subset(0,\infty)$ and satisfy \eqref{eq:korevaar-condition}.  Suppose
$u\in C^4(\Omega)\cap C^2(\overline\Omega)$ solves
\[
 a_{A,B}^{ij}(Du)u_{ij}=0,\qquad
 u|_{\partial\Omega_0}=0,\quad u|_{\partial\Omega_1}=1,
\]
with $|Du(x)|\in I$ for every $x\in\overline\Omega$.  Assume that the
intermediate level hypersurfaces are convex
and that the second fundamental form is positive definite at least at one point
of one intermediate level.  Then every intermediate level hypersurface is
strictly convex and
\[
 (-1)^{n-1}\det D^2u>0\qquad\text{in }\Omega.
\]
The Hessian has inertia $(1,n-1)$, and $Du$ is a global diffeomorphism onto the
radial ring \eqref{eq:gradient-ring}.
\end{corollary}

\begin{proof}
Korevaar's constant-rank theorem \cite{Korevaar} says that the second
fundamental forms of the level hypersurfaces have constant rank throughout the
connected ring.  Positivity at one point therefore forces that rank to be
$n-1$ everywhere.  The conclusion now follows from
\cref{thm:ring-intro}.
\end{proof}

\begin{remark}\label{rem:convexity-literature}
The hypotheses preceding strictness in \cref{cor:korevaar-ring} can themselves
be consequences of convex-ring PDE theory.  Korevaar's barrier, comparison, and
deformation arguments produce noncritical solutions with strictly convex levels
for operators satisfying his structural assumptions \cite{Korevaar}.  At the
macroscopic level, quasiconcave-envelope criteria were obtained by Bianchini,
Longinetti and Salani \cite{BianchiniLonginettiSalani}; at the microscopic
level, broader constant-rank and curvature estimates were established by Bian,
Guan, Ma and Xu and by Guan and Xu \cite{BianGuanMaXu,GuanXu}.  Our
contribution starts after strict convexity has been established and imposes no
further differential condition on $a(q)$.
\end{remark}

\section{The \texorpdfstring{$p$}{p}-Laplace and minimal surface equations}\label{sec:examples}

\subsection{The \texorpdfstring{$p$}{p}-Laplace equation}

 Let $1<p<\infty$ and suppose $|Du|>0$.  In $\R^n$ the equation
\[
 \operatorname{div}(|Du|^{p-2}Du)=0
\]
is equivalent, after division by $|Du|^{p-2}$, to
\[
 a_p^{ij}(Du)u_{ij}=0,
 \qquad
 a_p(q)=I+(p-2)\frac{q\otimes q}{|q|^2}.
\autotag\label{eq:ap}
\]
The eigenvalues of $a_p$ are
\[
 1\quad\text{with multiplicity }n-1,
 \qquad p-1\quad\text{in the radial direction}.
\]
The matrix is real analytic on $\R^n\setminus\{0\}$ and positive definite.
Specializing to $n=3$, \cref{thm:local-intro} gives the following result, recovering
\cite{LiLiuMa}.

\begin{corollary}[$p$-harmonic local Lewy theorem]\label{cor:p-local}
Let $u\in W^{1,p}_{\mathrm{loc}}(\Omega)$ be a weak $p$-harmonic function in
$\Omega\subset\R^3$ and suppose $|Du|>0$ in $\Omega$.  If $Du$ is a local
homeomorphism, then $\det D^2u$ is nowhere zero and $Du$ is a local
real-analytic diffeomorphism.  If $Du$ is globally injective, it is a global
real-analytic diffeomorphism onto $Du(\Omega)$.
\end{corollary}

\begin{proof}
The coefficient $a_p$ is real analytic and positive definite on
$\R^3\setminus\{0\}$.  Standard noncritical regularity followed by analytic
elliptic regularity makes $u$ real analytic locally in $\Omega$.  Apply
\cref{thm:local-intro} with $a=a_p$, and then use the analytic inverse function
theorem.
\end{proof}

For the capacitary problem the geometric hypotheses of
\cref{thm:ring-intro} need not be imposed on the solution.  Lewis' extension of
Gabriel's convexity theorem gives noncriticality and strict convexity of every
intermediate level for bounded convex rings \cite{Gabriel,Lewis}; see
\cite[Theorem~1.1]{MaOuZhang}.  Thus even strict convexity of the two boundary
components is unnecessary for the interior Lewy conclusion.

\begin{corollary}[$p$-capacitary gradient map]\label{cor:p-ring}
Let $1<p<\infty$ and $n\ge2$.  Let $\Omega_0$ and $\Omega_1$ be bounded smooth
convex domains in $\R^n$ with $\overline{\Omega_1}\Subset\Omega_0$, set
$\Omega=\Omega_0\setminus\overline{\Omega_1}$, and let $u$ be the
$p$-capacitary potential
\[
 \begin{cases}
 \operatorname{div}(|Du|^{p-2}Du)=0
     &\text{in }\Omega,\\
 u=0&\text{on }\partial\Omega_0,\\
 u=1&\text{on }\partial\Omega_1.
 \end{cases}
\]
Then
\[
 (-1)^{n-1}\det D^2u>0\qquad\text{in }\Omega,
\]
the Hessian has inertia $(1,n-1)$, and $Du$ is a global real-analytic
diffeomorphism onto its image.  More precisely, writing
$Du(x(t,\nu))=s(t,\nu)\nu$ on the level $u=t$, the limits
\[
 \rho_0(\nu)=\lim_{t\downarrow0}s(t,\nu),\qquad
 \rho_1(\nu)=\lim_{t\uparrow1}s(t,\nu)
\]
satisfy $0<\rho_0(\nu)<\rho_1(\nu)<\infty$ and
\[
 Du(\Omega)
 =\{r\nu:\nu\in\Sph^{n-1},\ \rho_0(\nu)<r<\rho_1(\nu)\}.
\]
\end{corollary}

\begin{proof}
By Lewis' convex-ring theorem, extending Gabriel's harmonic result,
$Du\ne0$ in the ring and every intermediate level hypersurface is strictly
convex.  The pointwise and global-injectivity
arguments in the proof of \cref{thm:ring-intro} use only those intermediate
levels, so they apply even though the boundary components need not be strictly
convex.  Formula \eqref{eq:s-monotone} shows that $s(t,\nu)$ is strictly
increasing and gives the asserted endpoint description.  Boundary regularity
and the Hopf boundary point lemma give a continuous nonvanishing gradient on
the compact set $\overline\Omega$.  Thus $|Du|$ is bounded above and bounded
away from zero there, which proves
$0<\rho_0(\nu)<\rho_1(\nu)<\infty$.  Finally, $a_p$ is real analytic away
from the origin; analytic elliptic regularity and the analytic inverse
function theorem show that the global diffeomorphism is real analytic.
\end{proof}

\begin{remark}
Lewis' qualitative theorem therefore supplies exactly the geometric hypotheses
needed here.  If the boundary components are strictly convex, the
Gaussian-curvature estimates of Ma, Ou and Zhang give quantitative positive
lower bounds in terms of boundary curvature and boundary gradient data
\cite[Theorem~1.2 and Corollary~1.3]{MaOuZhang}; such estimates supplement, but
are not required for, the nondegeneracy and injectivity conclusion above.
\end{remark}
\subsection{The minimal surface equation}

The minimal surface equation is
\[
 \operatorname{div}\left(\frac{Du}{\sqrt{1+|Du|^2}}\right)=0.
\]
After multiplication by $\sqrt{1+|Du|^2}$, it becomes
\[
 a_{\mathrm{ms}}^{ij}(Du)u_{ij}=0,
 \qquad
 a_{\mathrm{ms}}(q)
 =I-\frac{q\otimes q}{1+|q|^2}.
\autotag\label{eq:ams}
\]
The eigenvalues are
\[
 1\quad\text{with multiplicity }n-1,
 \qquad \frac1{1+|q|^2}\quad\text{in the radial direction}.
\]
Unlike \eqref{eq:ap}, the matrix \eqref{eq:ams} is real analytic and positive
definite in $\R^n$.  In dimension three no nonvanishing-gradient
assumption is needed in the local theorem.

\begin{corollary}[Minimal-surface local Lewy theorem]\label{cor:min-local}
Let $u\in C^{2,\alpha}_{\mathrm{loc}}(\Omega)$ solve the minimal surface equation
in a domain $\Omega\subset\R^3$.  If $Du$ is a local homeomorphism, then
\[
 \det D^2u\ne0
\]
throughout $\Omega$, and $Du$ is a local real-analytic diffeomorphism.  If $Du$
is globally injective, it is a global real-analytic diffeomorphism onto
$Du(\Omega)$.
\end{corollary}

\begin{proof}
The coefficient $a_{\mathrm{ms}}$ is real analytic and positive definite on all
of $\R^3$.  Analytic elliptic regularity gives $u\in C^\omega(\Omega)$.  The
nonvanishing determinant follows from \cref{thm:local-intro}, and the analytic inverse
function theorem gives the asserted analytic diffeomorphism statements.
\end{proof}
For a smooth solution of the ring problem below, noncriticality and strict
convexity of the level sets follow from Korevaar's convex-ring theory
\cite{Korevaar}; see also the extension to space forms by Wang and Zhang
\cite[Theorem~1.2]{WangDekaiZhang}.  These geometric properties therefore need
not be imposed separately.  Existence of a smooth solution is assumed:
solvability can depend on the boundary height difference and the geometry of
the ring.  The boundary values $0$ and $1$ are prescribed here; rescaling $u$
alone does not preserve the minimal surface equation.  Quantitative curvature
refinements may be found in \cite{MaYeYe,WangZhang,ZhangZhang}.

\begin{corollary}[Minimal graphs on convex rings]\label{cor:min-ring}
Let $n\ge2$ and let $\Omega_0$ and $\Omega_1$ be bounded smooth strictly convex
domains in $\R^n$ with $\overline{\Omega_1}\Subset\Omega_0$.  Set
$\Omega=\Omega_0\setminus\overline{\Omega_1}$.  Suppose that the minimal
surface ring problem
\[
 \begin{cases}
 \displaystyle\operatorname{div}\left(\dfrac{Du}{\sqrt{1+|Du|^2}}\right)=0
 &\text{in }\Omega,\\
 u=0&\text{on }\partial\Omega_0,\\
 u=1&\text{on }\partial\Omega_1
 \end{cases}
\]
admits a solution $u\in C^\infty(\Omega)\cap C^2(\overline\Omega)$.
Then this solution is unique, $Du\ne0$ in $\overline\Omega$, and every
intermediate level hypersurface is strictly convex.  Moreover,
\[
 (-1)^{n-1}\det D^2u>0\qquad\text{in }\Omega,
\]
the Hessian has inertia $(1,n-1)$, and $Du$ is a global real-analytic
diffeomorphism onto the radially parametrized ring \eqref{eq:gradient-ring}.
\end{corollary}

\begin{proof}
Uniqueness follows from the comparison principle.  For the assumed smooth
solution, the convex-ring theory gives noncriticality and strict convexity of
all intermediate level hypersurfaces \cite{Korevaar,WangDekaiZhang}.
The Hopf boundary point lemma gives $|Du|>0$ on $\partial\Omega$.
Apply \cref{thm:ring-intro} with $a=a_{\mathrm{ms}}$.
Since $a_{\mathrm{ms}}$ is real analytic in $\R^n$, analytic elliptic
regularity and the analytic inverse function theorem give the real-analytic
conclusion.
\end{proof}

\begin{samepage}
\begin{remark}
In dimension three, \cref{cor:min-ring} gives
\[
 \det D^2u>0.
\]
The restriction of $D^2u$ to the tangent space of each level hypersurface is
negative definite, while the full Hessian has inertia $(1,2)$.  In the
rotationally invariant case \eqref{eq:ams}, the equation also gives
\[
 u_{\nu\nu}=-(1+|Du|^2)\tr(D^2u|_{\nu^\perp})>0,
\]
but the general Schur-complement proof is stronger because it also controls the
mixed normal--tangential derivatives and works for anisotropic $a$.
\end{remark}
\end{samepage}

\section*{Acknowledgments}
The authors thank Professor Xi-Nan Ma for bringing this problem to their
attention.  This work was supported by the Natural Science Foundation of
Heilongjiang Province (Grant No. PL2025A006).  The authors acknowledge the use
of AI tools and take full responsibility for the content of the manuscript,
including all mathematical statements and proofs.

\end{document}